\documentclass[reqno,a4paper,11pt]{article}
\usepackage{amsmath,amsthm,dsfont,amsfonts,amssymb,fancyhdr,mathtools}
\usepackage{enumerate}
\usepackage[usenames,dvipsnames]{color}
\usepackage{bbm,soul,supertabular,longtable,verbatim,extarrows}
\usepackage{titlesec}
\usepackage{graphicx}
\usepackage[titletoc,toc,title]{appendix}

\usepackage{tikz}
\usetikzlibrary{arrows,decorations.pathmorphing,backgrounds,positioning,fit,petri}

\usepackage{float}
\usepackage{BOONDOX-cal}
\usepackage{cite}
\usepackage{booktabs,multirow,makecell}
\usepackage{authblk}
\usetikzlibrary{trees}
\usepackage{lipsum}
\usepackage{mathrsfs}
\usepackage[top=2.4cm,bottom=2.2cm,left=2.6cm,right=2cm]{geometry}
\usepackage{comment}
\usepackage[T1]{fontenc}
\usepackage[utf8]{inputenc}
\usepackage{lmodern}
\usepackage{microtype}
\usepackage{aliascnt}
\usepackage{enumitem}
\usepackage{indentfirst}

\newcommand{\G}{\Gamma}
\newcommand{\lmin}{\lambda_{\min}}

\newcommand{\Sp}{\operatorname{Sp}}

\newcommand{\dist}{\operatorname{dist}}
\newcommand{\Z}{\mathbb{Z}}
\newcommand{\Nslim}{N^{\mathrm{slim}}}
\newcommand{\Nfat}{N^{\mathrm{fat}}}
\newcommand{\Vslim}{V_{\mathrm{slim}}}
\newcommand{\Vfat}{V_{\mathrm{fat}}}

\newcommand{\Sh}{\Gamma_{\operatorname{Sh}}}

\newcommand{\ip}[2]{\left(#1,#2\right)}

\newtheorem{thm}{Theorem}[section]

\newtheorem{de}[thm]{Definition}

\newtheorem{theorem}[thm]{Theorem}
\newtheorem{lemma}[thm]{Lemma}
\newtheorem{proposition}[thm]{Proposition}

\newtheorem{assumption}[thm]{Assumption}

\soulregister\cite7
\soulregister\citep7
\soulregister\citet7
\soulregister\ref7
\soulregister\pageref7
\sethlcolor{yellow}
\setstcolor{red}
\soulregister{\em}{0}

\usepackage[colorlinks=true, allcolors=blue]{hyperref}
\usepackage[notref]{showkeys} 

\begin{document}
	\title{The integrability of sesqui-regular graphs with smallest eigenvalue at least $-3$}
	\author[a,b]{Qianqian Yang}
	\author[a]{Ruo-Lan Ding}
	\affil[a] {\footnotesize{Department of Mathematics, Shanghai University, Shanghai 200444, PR China}}
	\affil[b] {\footnotesize{Newtouch Center for Mathematics of Shanghai University, Shanghai 200444, PR China}}

	\maketitle
	\pagestyle{plain}
	
	\newcommand\blfootnote[1]{%
		\begingroup
		\renewcommand\thefootnote{}\footnote{#1}%
		\addtocounter{footnote}{-1}%
		\endgroup}
	\blfootnote{E-mail addresses:  {\tt qqyang@shu.edu.cn} (Q. Yang),  {\tt rlding@shu.edu.cn} (R.-L. Ding).}

\begin{abstract}
In this paper, we prove that every connected sesqui-regular graph with parameters $(n,k,c)$, where $c\geq3$, and with smallest eigenvalue in $[-3,-2)$ is
$1$-integrable if its valency is sufficiently large.  The bound on $c$ is optimal: for every integer $t\geq2$, the Cartesian product of the Shrikhande graph and $K_t$ is a connected
sesqui-regular graph with parameters $(16t,t+5,2)$ and smallest eigenvalue
$-3$, but it is not $1$-integrable.  
\end{abstract}



\section{Introduction}\label{sec:introduction}

All graphs are finite, undirected, and simple.  For a graph $G$, let
$V(G)$ and $E(G)$ denote its vertex and edge sets, respectively, and let
$A(G)$ denote its adjacency matrix.  The smallest eigenvalue of $A(G)$ is
denoted by $\lmin(G)$.  For $x,y\in V(G)$, write $x\sim_G y$ if $x$ and $y$ are adjacent, and put
\[
 N_G(x)=\{y\in V(G):x\sim_G y\}.
\]
Write
$\dist_G(x,y)$ for the distance between $x$ and $y$. The $i$th distance neighborhood of $x$ is
\[
 G_i(x)=\{y\in V(G):\dist_G(x,y)=i\}.
\]
The \emph{valency} of $x$ is $k_G(x)=|N_G(x)|$.  For a subset $X\subseteq V(G)$, write $G[X]$ for the subgraph of $G$ induced on $X$. For disjoint subsets
$X,Y\subseteq V(G)$, let $e_G(X,Y)$ denote the number of edges with one endpoint in $X$ and the other in $Y$; the graph subscript is omitted when the ambient graph is clear.

The graph $G$ is $k$-regular if every vertex has valency $k$.  A $k$-regular graph
on $n$ vertices is \emph{strongly regular} with parameters $(n,k,a,c)$ if
each pair of adjacent vertices has exactly $a$ common neighbors and each 
pair of nonadjacent vertices has exactly $c$ common neighbors.

A $k$-regular graph on $n$ vertices is \emph{sesqui-regular with parameters
$(n,k,c)$} if
\[
|N_G(x)\cap N_G(y)|=c
 \quad\text{whenever}~
 \dist_G(x,y)=2.
\]
Thus sesqui-regularity prescribes the number of common neighbors of
vertices at distance two, without imposing a corresponding condition on adjacent
vertices. A sesqui-regular graph of diameter two is therefore
\emph{co-edge-regular}. In particular, the class of sesqui-regular graphs
contains all strongly regular graphs and co-edge-regular graphs and is substantially broader.

For adjacent vertices $x$ and $y$, let $a(x,y)$ denote the number of their
common neighbors. The \emph{level} $\ell(\Gamma)$ of a sesqui-regular graph
$\Gamma$ is defined by
\[
\ell(\Gamma):=\left|\{a(x,y)\mid x\sim_\G y\}\right|.
\]
Every strongly regular graph has level $1$. Recent results show that co-edge-regular graphs can remain structurally complicated even when their spectra is highly restricted. Ge and Koolen
\cite{Ge.koolen2025co.edge.regular} constructed infinite families of
co-edge-regular graphs of level at least $3$ with exactly four distinct
eigenvalues. More recently, van Dam, Ge, and Koolen
\cite{vanDam.Ge.Koolen2026} constructed, for every prime power $q$, infinitely many co-edge-regular
graphs with exactly four distinct eigenvalues, smallest eigenvalue
$-2q-1$, and level at least $q+2$. Thus, even within this restrictive
four-eigenvalue setting, the level can tend to infinity as the smallest
eigenvalue tends to negative infinity. 

These results illustrate that the structure of sesqui-regular graphs can remain highly complicated under strong spectral restrictions, and that spectral data alone do not generally lead to an effective classification.

For the graphs considered in this paper, the matrix $A+3I$ is positive semidefinite, where $A$ denotes the adjacency matrix. It is therefore natural to ask whether this matrix admits an integral Gram representation, which leads to the following notion.

For a positive integer $s$, a graph $G$ is \emph{$s$-integrable} if its
vertices can be assigned integral vectors $x\mapsto\mathbf{x}$ such that
\begin{equation}\label{def:s-integrable}
	(\mathbf{x},\mathbf{y})=\left\{
	\begin{array}{ll}
		s\lceil-\lambda_{\min}(G)\rceil & \text{ if }x=y, \\
		s & \text{ if }x\sim y, \\
		0 & \text{ otherwise},
	\end{array}
	\right.	
\end{equation}
where $(\cdot,\cdot)$ denotes the standard inner product. Equivalently,
$G$ is $s$-integrable if and only if there exists an integral matrix $N$
such that
\[
 A(G)+\lceil-\lambda_{\min}(G)\rceil I
 =\frac1sN^{\mathsf T}N.
\]

Graphs with a fixed lower bound on the smallest eigenvalue often
exhibit strong structural restrictions when the valency is large. Koolen,
Yang, and Yang \cite{KoolenYangYang2018} proved the following general
integrability result.
\begin{thm}[{\cite[Theorem~1.3]{KoolenYangYang2018}}]\label{thm:kyy-2-integrable}
	There exists a positive integer $\kappa_{1}$ such that every
	connected graph $\Gamma$ with
	$\lambda_{\min}(\Gamma)\geq-3$ and minimum valency at least
	$\kappa_{1}$ is $s$-integrable, for every integer $s\geq2$.
\end{thm}

The same work also provides infinite families of non-$1$-integrable graphs
with smallest eigenvalue at least $-3$ and unbounded minimum valency.
Consequently, without additional hypotheses, Theorem~\ref{thm:kyy-2-integrable}
cannot be strengthened in general from $2$-integrability to
$1$-integrability.

Within the class of sesqui-regular graphs, Yang, Gebremichel, Rehman,
Yang, and Koolen~\cite{YangEtAl2021} determined the graphs corresponding
to the two extremal parameter values $c=k$ and $c=k-1$.

\begin{thm}[{\cite[Theorem~1.6]{YangEtAl2021}}]\label{thm:extreme}
	Let $\Gamma$ be a connected sesqui-regular graph with parameters
	$(n,k,c)$ and smallest eigenvalue
	$\lambda_{\min}(\Gamma)\in[-3,-2)$. If $c=k$, then $n\geq6$ and $\Gamma$ is the
	complete multipartite graph $K_{3,\ldots,3}$. If $c=k-1$,
	then $n\geq7$ and $\Gamma$ is either the complement of a disjoint union of cycles of
	length at least $4$, or the $3$-cube. In particular, $\Gamma$ is
	$1$-integrable.
\end{thm}
The remaining values of $c$ are constrained by the following result of
Koolen, Gebremichel, Yang, and Yang. It shows that, when the smallest
eigenvalue is bounded below and the valency is sufficiently large, either
$c$ is bounded or the order of the graph is close to its valency.

\begin{theorem}[{\cite[Theorem~1.5]{KoolenGebremichelYangYang2023} }]
\label{thm:large-valency-alternative}
For every integer $\lambda\ge2$, there exists a positive integer
$\kappa_2(\lambda)$ such that, if $\G$ is a connected sesqui-regular graph
with parameters $(n,k,c)$, $\lmin(\G)\ge-\lambda$, and
$k\ge\kappa_2(\lambda)$, then
\[
 c\le\lambda^2(\lambda-1)
 \quad\text{or}\quad
 n-k-1\le\frac{(\lambda-1)^2}{4}+1.
\]
\end{theorem}

Applying Theorem~\ref{thm:large-valency-alternative} with $\lambda=3$
shows that, for a connected sesqui-regular graph with smallest eigenvalue at
least $-3$ and sufficiently large valency, either $c\leq18$ or
$n-k-1\leq2$. In the range $c\geq9$, Yang et al.~\cite{YangEtAl2021} obtained the following stronger conclusion for sesqui-regular graphs whose smallest eigenvalue lies in $[-3,-2)$.

\begin{thm}[{\cite[Theorem~1.7 and~Theorem~1.8]{YangEtAl2021}}]\label{thm:cgeq9}
	There exists a positive integer $\kappa_{3}$ such that every
	connected sesqui-regular graph $\Gamma$ with parameters $(n,k,c)$,
	with $c\geq9$, smallest eigenvalue
	$\lambda_{\min}(\Gamma)\in[-3,-2)$, and $k\geq\kappa_{3}$,
	is $1$-integrable. Moreover, if $c\leq k-2$, then $c=9$ and
	$\Gamma$ is the block graph of a Steiner triple system.
\end{thm}

In this paper, we mainly investigate the integrability of sesqui-regular graphs with smallest eigenvalue in $[-3,-2)$ and with parameter $c\in\{3,4,5,6,7,8\}$. Our main result is as follows.

\begin{thm}\label{thm:main}
There exists a positive integer $\kappa_{4}$ such that for every
connected sesqui-regular graph $\G$ with parameters $(n,k,c)$ and smallest eigenvalue $\lmin(\G)\in[-3,-2)$, if 
\[
 c\geq3 \quad\text{and}\quad k\geq\kappa_4,
\]
then $\G$ is $1$-integrable.
\end{thm}

Theorem~\ref{thm:main} extends the known range $c\geq9$ to the optimal
threshold $c\geq3$. In particular, it shows that sesqui-regularity with
$c\geq3$ provides enough additional structure to strengthen the general
large valency conclusion from $2$-integrability to $1$-integrability.
Theorem~\ref{thm:sharpness} shows that the lower bound on $c$ is best possible.
For graphs $G$ and $H$, their
\emph{Cartesian product} $G\square H$ has vertex set $V(G)\times V(H)$,
where $(x,i)$ and $(y,j)$ are adjacent precisely when either $x=y$ and
$i\sim_H j$, or $i=j$ and $x\sim_G y$.  Let $\Sh$ denote the Shrikhande graph, which is the strongly regular graph with parameters $(16,6,2,2)$.

\begin{thm}\label{thm:sharpness}
For every positive integer $t\geq2$, the graph 
\[
 \G_t=\Sh \square K_t
\]
is connected and sesqui-regular with parameters
\[
 (n,k,c)=(16t,t+5,2).
\]
Moreover, $\lmin(\G_t)=-3$, and $\G_t$ is not $1$-integrable.
\end{thm}

The remainder of the paper is organized as follows. Section~\ref{sec:preliminaries}
reviews the required material on Hoffman graphs and
integral representations. Section~\ref{sec:proof-main} proves
Theorem~\ref{thm:main} by reducing a hypothetical non-integrable Hoffman
factor to its special graph and treating the ranges $c\geq5$ and
$c\in\{3,4\}$ separately. Section \ref{sec:sharpness} proves
Theorem~\ref{thm:sharpness}.

\section{Preliminaries}\label{sec:preliminaries}

This section collects the definitions and results on
Hoffman graphs and integral representations used in the proof of Theorem~\ref{thm:main}.

\subsection{Hoffman graphs}
\begin{de}A \emph{Hoffman graph} $\mathfrak{h}$ is a pair $(H,\ell)$, where $H=(V(H),E(H))$ is a graph and $\ell:V(H) \to\{f,s\}$ is a labeling map satisfying the following conditions:
	\begin{enumerate}
		\item vertices with label $f$ are pairwise non-adjacent,
		\item every vertex with label $f$ is adjacent to at least one vertex.
	\end{enumerate}
\end{de}
A vertex with label $s$ is called a \emph{slim vertex}, and a vertex with label $f$
is called a \emph{fat vertex}. The sets of slim and fat vertices are denoted by 
$\Vslim(\mathfrak h)$ and
$\Vfat(\mathfrak h)$, respectively. The subgraph
induced by $\Vslim(\mathfrak h)$ is called the \emph{slim graph} of $\mathfrak{h}$.  

For a slim vertex $x$, write $\Nfat_{\mathfrak h}(x)$ for the set of its
fat neighbors, and put
$\Nfat_{\mathfrak h}(x,y)=\Nfat_{\mathfrak h}(x)\cap
\Nfat_{\mathfrak h}(y)$.  A Hoffman graph is \emph{fat} when every slim
vertex has a fat neighbor.  For $f\in\Vfat(\mathfrak h)$, the
\emph{quasi-clique} $Q_{\mathfrak h}(f)$ is the subgraph of the slim graph
induced by the slim neighbors of $f$.

The \emph{special matrix} of $\mathfrak h$ is indexed by its slim vertices and is
defined by
\begin{equation}\label{eq:special-matrix}
 \Sp(\mathfrak h)_{xy}=
 \begin{cases}
 -|\Nfat_{\mathfrak h}(x)|,&x=y,\\
 1-|\Nfat_{\mathfrak h}(x,y)|,&x\sim y,\\
 -|\Nfat_{\mathfrak h}(x,y)|,&x\not\sim y.
 \end{cases}
\end{equation}
The \emph{eigenvalues} of $\mathfrak h$ are those of $\Sp(\mathfrak h)$, and the smallest one is denoted by $\lambda_{\min}(\mathfrak{h})$.

The \emph{special graph} of $\mathfrak h$ is the signed graph on
$\Vslim(\mathfrak h)$ in which two distinct vertices are joined by a
positive or negative edge according to the sign of the corresponding
off-diagonal entry of $\Sp(\mathfrak h)$.

A Hoffman graph $\mathfrak{h}_1:= (H_1, \ell_1)$ is an \emph{induced}
(respectively, \emph{proper induced}) \emph{Hoffman subgraph} of
$\mathfrak{h}:=(H, \ell)$ if $H_1$ is an induced (respectively, proper
induced) subgraph of $H$ and $\ell_1(x) = \ell(x)$ for every vertex $x$ of
$H_1$.

\vspace{0.1cm}

The following interlacing result for induced Hoffman subgraphs is due to Woo and Neumaier.

\begin{lemma}[{\cite[Corollary~3.3]{WooNeumaier1995}}]
\label{lem:hoffman-interlacing}
If $\mathfrak h_1$ is an induced Hoffman subgraph of $\mathfrak h$, then
\[
 \lmin(\mathfrak h_1)\ge\lmin(\mathfrak h).
\]
\end{lemma}

For Hoffman graphs $\mathfrak h_1$ and $\mathfrak h_2$, their \emph{sum},
denoted by $\mathfrak h_1\uplus\mathfrak h_2$, is the Hoffman graph whose
slim vertex set is $\Vslim(\mathfrak h_1)\cup\Vslim(\mathfrak h_2)$ and
whose special matrix, after the corresponding ordering of the slim
vertices, is the block-diagonal matrix
\[
\begin{pmatrix}
    \Sp(\mathfrak h_1)&\\
    &\Sp(\mathfrak h_2)
\end{pmatrix}.
\]
If $\mathfrak{h}=\mathfrak{h}_1\uplus\mathfrak{h}_2$ for two nonempty
Hoffman subgraphs $\mathfrak h_1$ and $\mathfrak h_2$, then
$\mathfrak h$ is \emph{decomposable}, and $\mathfrak h_1$ and
$\mathfrak h_2$ are called \emph{factors} of $\mathfrak h$. Otherwise,
$\mathfrak h$ is \emph{indecomposable}. A Hoffman graph is
indecomposable if and only if its special graph is connected.

The sum of two Hoffman graphs has the following combinatorial
description.

\begin{lemma}[{\cite[Lemma~2.11]{KoolenYangYang2019}}]
\label{lem:sum-criterion}
Suppose that $\mathfrak h=\mathfrak h_1\uplus\mathfrak h_2$.  If
$x\in\Vslim(\mathfrak h_1)$ and
$y\in\Vslim(\mathfrak h_2)$, then $x,y$ have at most one common fat
neighbor, and they are adjacent in the slim graph if and only if they have
exactly one common fat neighbor.  Moreover, every fat neighbor in
$\mathfrak h$ of a slim vertex of $\mathfrak h_i$ belongs to
$\mathfrak h_i$.
\end{lemma}

\subsection{Representations and reduced representations of Hoffman graphs}
An \emph{integral representation of norm $t$} of a Hoffman graph
$\mathfrak h$ is a map
$\boldsymbol{\phi}:\Vslim(\mathfrak h)\cup\Vfat(\mathfrak h)\to\Z^m$ for some $m$
such that
\[
 \left(\boldsymbol{\phi}(x),\boldsymbol{\phi}(y)\right)=
 \begin{cases}
 t,&x=y\in\Vslim(\mathfrak h),\\
 1,&x=y\in\Vfat(\mathfrak h),\\
 1,&x\sim y,\\
 0,&\text{otherwise}.
 \end{cases}
\]
An \emph{integral reduced representation of norm $t$} of  
$\mathfrak h$ is a map
$\boldsymbol{\psi}:\Vslim(\mathfrak h)\to\Z^m$ for some $m$
such that
	\[
	(\boldsymbol{\psi}(x),\boldsymbol{\psi}(y))=\left\{
	\begin{array}{ll}
		t-|N_\mathfrak{h}^{\mathrm{fat}}(x)| & \text{if } x=y, \\
		1-|N_\mathfrak{h}^{\mathrm{fat}}(x,y)| & \text{if } x\sim y, \\
		-|N_\mathfrak{h}^{\mathrm{fat}}(x,y)| & \text{otherwise},
	\end{array}
	\right.
	\]

The following two results are due to Jang, Koolen, Munemasa and Taniguchi. 

\begin{theorem}[{\cite[Theorem~2.8]{JangKoolenMunemasaTaniguchi2014}}]
\label{thm:representation-equivalence}
A Hoffman graph has an integral representation of norm $t$ if and only if it
has an integral reduced representation of norm $t$.
\end{theorem}

\begin{theorem}[{\cite[Theorem~3.7]{JangKoolenMunemasaTaniguchi2014}}]
\label{thm:two-fat-integral}
Let $\mathfrak g$ be a fat indecomposable Hoffman graph with smallest
eigenvalue at least $-3$.  If one of its slim vertices has at least two fat
neighbors, then $\mathfrak g$ has an integral reduced representation of norm
$3$.
\end{theorem}

\section{Proof of Theorem~\ref{thm:main}}\label{sec:proof-main}
This section proves Theorem~\ref{thm:main}. Since eigenvalue interlacing for induced subgraphs will be used repeatedly, the required form is stated as follows. 

\begin{lemma}[{\cite[Theorem~9.1.1]{GodsilRoyle2001}}]
\label{lem:interlacing}
If $G$ is a graph and $G'$ is an induced subgraph of $G$, then
\[
 \lambda_{\min}(G')\geq\lambda_{\min}(G).
\]
\end{lemma}

The next theorem provides the bridge from sesqui-regular graphs to fat Hoffman graphs while preserving the lower bound on the smallest eigenvalue. 

\begin{theorem}[{\cite[Lemma~5.1]{KoolenYangYang2018} and \cite[Theorem~4.1]{KoolenGebremichelYangYang2023}}]\label{thm:hoffman-cover}
There exists a positive integer $\kappa_5$ such that the following holds.
Let $\G$ be a connected sesqui-regular graph with parameters $(n,k,c)$,
$n-k-1>2$, and $\lmin(\G)\ge-3$.  If $k\ge\kappa_5$, then there exists a
fat Hoffman graph $\mathfrak h$ such that
\begin{enumerate}[label=\textup{(\roman*)}]
 \item the slim graph of $\mathfrak h$ is $\G$;
 \item $\lmin(\mathfrak h)\ge-3$;
 \item every quasi-clique of $\mathfrak h$ is a clique.
\end{enumerate}
\end{theorem}

Theorem~\ref{thm:cgeq9}
reduces the remaining part of the proof of Theorem~\ref{thm:main} to the range
$3\leq c\leq8$. The following standing assumption will be used throughout the  structural analysis. 

\begin{assumption}\label{ass:Gamma-h}
The graph $\G$ is a connected sesqui-regular graph with parameters
$(n,k,c)$ satisfying
\begin{equation}\label{eq:parameter-range}
 3\le c\le8,\qquad -3\le\lmin(\G)<-2,\qquad k\ge98.
\end{equation}
Moreover, $\mathfrak h$ is a fat Hoffman graph satisfying the conclusions
of Theorem~\ref{thm:hoffman-cover}; namely, its slim graph is $\G$,
$\lmin(\mathfrak h)\geq-3$, and every quasi-clique of $\mathfrak h$ is a
clique.
\end{assumption}

The following lemma gives an elementary bound on intersections of quasi-cliques.
\begin{lemma}[{\cite[Fact~1]{YangEtAl2021}}]\label{lem:quasi-bounds}
Under Assumption~\ref{ass:Gamma-h}, if $f$ and $f'$ are two distinct fat
vertices of $\mathfrak h$, then
$|\Nslim_{\mathfrak h}(f,f')|\le2$.
\end{lemma}

\subsection{Structure of indecomposable factors with no integral reduced representation of norm 3}

To prove Theorem~\ref{thm:main}, it suffices to show that the Hoffman
graph $\mathfrak h$ has an integral representation of norm $3$. By
Theorem~\ref{thm:representation-equivalence}, this is equivalent to
show that $\mathfrak h$ has an integral reduced representation of norm
$3$. The remainder of this subsection analyzes an indecomposable factor for which such an integral reduced representation does not exist. 

\begin{assumption}\label{ass:exceptional-factor}
Under Assumption~\ref{ass:Gamma-h}, let \(\mathfrak g\) be an
indecomposable factor of \(\mathfrak h\) that admits no integral reduced
representation of norm \(3\), and let \(D\) be its special graph.
\end{assumption}

For each
\(u\in V(D)\), once uniqueness has been established below, let \(f_u\)
denote the unique fat neighbor of \(u\) in \(\mathfrak g\). For each fat
vertex \(f\) of \(\mathfrak g\), put
\[
 F_f:=\{u\in V(D):f_u=f\}
\]
and 
\[
  C(f):=V\bigl(Q_{\mathfrak h}(f)\bigr)
       =N^{\mathrm{slim}}_{\mathfrak h}(f).
\]

The following classical result of Cameron, Goethals, Seidel, and Shult will be used to bound the order of special graphs.

\begin{theorem}[{\cite{CameronGoethalsSeidelShult1976}}]\label{thm:large-minus-two-integrable}
A connected graph with smallest eigenvalue at least $-2$ and more than
$36$ vertices is $1$-integrable.
\end{theorem}

\begin{lemma}\label{lem:special-graph-of-g}
Under Assumptions~\ref{ass:Gamma-h}
and~\ref{ass:exceptional-factor}, the following hold.
\begin{enumerate}
\item Every slim vertex of
      \(\mathfrak g\) has exactly one fat neighbor.
\item The special graph \(D\) of \(\mathfrak g\) has the following properties. 
\begin{enumerate}
    \item It is connected;
     \item It has no negative edges and hence may be regarded as an ordinary graph;
    \item  
     \begin{equation}
        A(D)=\Sp(\mathfrak g)+I\qquad\text{and}\qquad \lambda_{\min}(D)\geq -2,
        \label{eq:Sp-D}
      \end{equation}
      where $A(D)$ is the adjacency matrix of $D$;
    \item It is not $1$-integrable and
      \begin{equation}
        |V(D)|\leq36;
        \label{eq:D-size}
      \end{equation}
    \item The nonempty sets \(F_f\) form a partition of \(V(D)\), and each of them is independent in \(D\).
\end{enumerate} 
\end{enumerate}
\end{lemma}

\begin{proof}
Throughout the proof, Assumptions~\ref{ass:Gamma-h}
and~\ref{ass:exceptional-factor} are in force.

{\rm (i)} Since \(\mathfrak g\) is a factor of the fat Hoffman graph
\(\mathfrak h\), Lemma~\ref{lem:sum-criterion} implies that every fat
neighbor in \(\mathfrak h\) of a slim vertex of \(\mathfrak g\) also belongs
to \(\mathfrak g\). Thus \(\mathfrak g\) is fat. By Lemma~\ref{lem:hoffman-interlacing},
\[
  \lambda_{\min}(\mathfrak g)
  \geq \lambda_{\min}(\mathfrak h)
  \geq -3.
\]
If some slim vertex of \(\mathfrak g\) had at least two fat neighbors,
Theorem~\ref{thm:two-fat-integral} would give an integral reduced
representation of norm \(3\) for \(\mathfrak g\), contrary to
Assumption~\ref{ass:exceptional-factor}. Since \(\mathfrak g\) is fat, every
slim vertex
therefore has exactly one fat neighbor.

{\rm (ii) ($a$)} The indecomposability of \(\mathfrak g\) is equivalent to
connectedness of its special graph, so \(D\) is connected.

{\rm ($b$)}  Let \(x\) and \(y\) be distinct slim vertices of \(\mathfrak g\). A
negative special edge could occur only if \(x\) and \(y\) were nonadjacent
in the slim graph and shared a fat neighbor. In that case both slim vertices
would lie in the same quasi-clique of \(\mathfrak h\), which is a clique by
Assumption~\ref{ass:Gamma-h}, a contradiction. Thus the special graph has
no negative edges.

{\rm ($c$)} Since every slim vertex has
exactly one fat neighbor, each diagonal entry of
\(\operatorname{Sp}(\mathfrak g)\) equals \(-1\), while an off-diagonal
entry equals \(1\) precisely on an edge of \(D\). This proves
$A(D)=\Sp(\mathfrak g)+I$, and it follows immediately that
\[
  \lambda_{\min}(D)
  =\lambda_{\min}(\mathfrak g)+1
  \geq -2.
\]

{\rm ($d$)} Suppose, to the contrary, that $D$ is $1$-integrable. If
$D\cong K_1$, then $\mathfrak{g}$ has a single slim vertex with one fat neighbor. Assigning the integral vector $\mathbf{e}_1+\mathbf{e}_2$ to the slim vertex gives $\mathfrak{g}$ an integral reduced representation of norm $3$, a
contradiction. Thus $|V(D)|\geq2$. 

If $\lambda_{\min}(D)=-1$, then $D$ is a clique, because every connected
graph with smallest eigenvalue at least $-1$ is complete. Assigning
$\mathbf e_0+\mathbf e_i$ to the $i$th vertex of $D$ gives an integral matrix
$N$ such that
\begin{equation}\label{eq:non-1-integrable}
 A(D)+2I=N^{\mathsf T}N.
\end{equation}
If $-2\leq\lambda_{\min}(D)<-1$, the assumed $1$-integrability of $D$
again gives an integral matrix $N$ satisfying
\eqref{eq:non-1-integrable}. Since
$A(D)+2I=\operatorname{Sp}(\mathfrak g)+3I$, the columns of $N$ give an
integral reduced representation of norm $3$ for $\mathfrak g$, a
contradiction. Hence $D$ is not $1$-integrable. Since $D$ is connected
and has smallest eigenvalue at least $-2$,
Theorem~\ref{thm:large-minus-two-integrable} now implies
\eqref{eq:D-size}.

{\rm ($e$)} The sets \(F_f\) clearly partition \(V(D)\). For two distinct vertices
\(x,y\in F_f\), they lie in the clique \(Q_{\mathfrak h}(f)\), and
hence they are adjacent in the slim graph. Their corresponding special
matrix entry is nevertheless \(1-1=0\), because they share the fat vertex
\(f\). Thus \(\{x,y\}\notin E(D)\), and every set \(F_f\) is independent.
\end{proof}

\begin{lemma}
\label{lem:large-quasi-cliques}
Under Assumptions~\ref{ass:Gamma-h}
and~\ref{ass:exceptional-factor}, for every $u\in V(D)$,
\begin{equation}
 N_{\Gamma}(u)
 =\bigl(C(f_u)\setminus\{u\}\bigr)\,\dot\cup\,N_D(u),
 \qquad
 k=|C(f_u)|-1+d_D(u).
 \label{eq:ambient-valency}
\end{equation}
In particular,
\begin{equation}\label{eq:clique order}
 |C(f_u)|\geq k-34.
\end{equation}
\end{lemma}

\begin{proof}
Let $v\sim_{\Gamma}u$. If
$v\notin V_{\mathrm{slim}}(\mathfrak g)$, then $u$ and $v$ belong to
different factors of $\mathfrak h$. Lemma~\ref{lem:sum-criterion}
therefore implies that their unique common fat neighbor is $f_u$, so
$v\in C(f_u)$. If $v\in V_{\mathrm{slim}}(\mathfrak g)$, then either
$f_v=f_u$, in which case $v\in C(f_u)$, or $f_v\neq f_u$, in which case
$v\in N_D(u)$.

Conversely, every vertex of $C(f_u)\setminus\{u\}$ is adjacent to $u$
because $C(f_u)$ is a clique, and every vertex of $N_D(u)$ is adjacent
to $u$ in $\Gamma$. The two sets are disjoint: the vertices of $C(f_u)\cap\Vslim(\mathfrak{g})$ belong to $F_{f_u}$, whereas $N_D(u)$ is disjoint from $F_{f_u}$. This proves \eqref{eq:ambient-valency}. Since
$d_D(u)\leq |V(D)|-1\leq35$, \eqref{eq:clique order} follows.
\end{proof}

\begin{lemma}
\label{lem:ordered-edge-sets}
Under Assumptions~\ref{ass:Gamma-h}
and~\ref{ass:exceptional-factor}, let $x$ and $y$ be two adjacent vertices in $D$, and put
\[
 C_x:=C(f_x),\qquad C_y:=C(f_y),\qquad \mathcal F_{x,y}:=\{f_u:u\in N_D(x),\ f_u\neq f_y\},
\]
and
\begin{equation}
 Z_{x,y}:=C_y\setminus
 \left(
 V_{\mathrm{slim}}(\mathfrak g)
 \cup
 \bigcup_{f\in\mathcal F_{x,y}\cup\{f_x\}}C(f)
 \right).
 \label{eq:Zxy-set}
\end{equation}
Then
\begin{equation}
 C_x\cap C_y=\varnothing,
 \qquad |\mathcal F_{x,y}|\leq34,
 \label{eq:ordered-edge-set-bounds}
\end{equation}
\begin{equation}
 Z_{x,y}=C_y\setminus V_{\mathrm{slim}}(\mathfrak g)
          =C_y\setminus F_{f_y}, \label{eq:ordered-edge-set-bounds2}
\end{equation}
and 
\begin{equation}
 |Z_{x,y}|=k+1-d_D(y)-|F_{f_y}|\geq k-35.
 \label{eq:ordered-edge-set-bounds3}
\end{equation}
For every $z\in Z_{x,y}$, the vertices $x$ and $z$ are at distance two
in $\Gamma$. Put
\[
 W_x:=N_{\Gamma}(z)\cap C_x,
 \qquad w_x:=|W_x|.
\]
Then
\begin{equation}
 N_{\Gamma}(x,z)
 =W_x\,\dot\cup\,\bigl(N_D(x)\cap F_{f_y}\bigr),
 \label{eq:common-neighbour-decomposition}
\end{equation}
and consequently
\begin{equation}
 |N_D(x)\cap F_{f_y}|=c-w_x.
 \label{eq:internal-common-count}
\end{equation}
\end{lemma}

\begin{proof}
The independence of the sets $F_f$ gives $f_x\neq f_y$. More generally,
if $f$ and $f'$ are distinct fat vertices of $\mathfrak g$, then
\begin{equation}
 C(f)\cap C(f')=\varnothing.
 \label{eq:disjoint-quasi-cliques}
\end{equation}
Indeed, the sets $F_f$ and $F_{f'}$ are nonempty because every fat
vertex of $\mathfrak g$ has a slim neighbor. Lemma~\ref{lem:quasi-bounds}
gives $|C(f)\cap C(f')|\leq2$, while
Lemma~\ref{lem:large-quasi-cliques}, applied to vertices of $F_f$ and
$F_{f'}$, gives $|C(f)|,|C(f')|\geq k-34$. If the intersection were nonempty, each
vertex in it would have valency at least
\[
 |C(f)\cup C(f')|-1\geq2(k-34)-2-1=2k-71>k,
\]
a contradiction. In particular, $C_x\cap C_y=\varnothing$.

Because $y\in N_D(x)\cap F_{f_y}$ and $|V(D)|\leq36$,
\[
 |\mathcal F_{x,y}|\leq |N_D(x)|-1\leq34.
\]
Every fat vertex in $\mathcal F_{x,y}\cup\{f_x\}$ is distinct from
$f_y$, so \eqref{eq:disjoint-quasi-cliques} implies \eqref{eq:ordered-edge-set-bounds2}. By \eqref{eq:ambient-valency} and the independence of $F_{f_y}$,
\[
 |Z_{x,y}|
 =k+1-d_D(y)-|F_{f_y}|
 \geq k+1-|V(D)|
 \geq k-35.
\]
This proves \eqref{eq:ordered-edge-set-bounds3}.

Let $z\in Z_{x,y}$. Since $y,z\in C_y$, the vertices $y$ and $z$ are
adjacent in $\Gamma$. Moreover, $z\notin V_{\mathrm{slim}}(\mathfrak g)$ by \eqref{eq:ordered-edge-set-bounds2}.
If $x\sim_{\Gamma}z$, then Lemma~\ref{lem:sum-criterion} would imply
$z\in C_x$, contrary to $C_x\cap C_y=\varnothing$. Hence
$\operatorname{dist}_{\Gamma}(x,z)=2$.

First,
\begin{equation}
 N_{\Gamma}(x,z)\setminus V_{\mathrm{slim}}(\mathfrak g)=W_x.
 \label{eq:outside-g-neighbours}
\end{equation}
Indeed, every vertex of $W_x$ is adjacent to $x$ because $C_x$ is a
clique. No vertex of $W_x$ belongs to $V_{\mathrm{slim}}(\mathfrak g)$:
otherwise its unique fat neighbor would be $f_x$, and its adjacency to
$z$ across two factors would imply $z\in C_x$. Conversely, if
$v\in N_{\Gamma}(x,z)\setminus V_{\mathrm{slim}}(\mathfrak g)$, then
the adjacency of $x$ and $v$ across two factors implies that their
common fat neighbor is $f_x$. Hence $v\in C_x$, and therefore
$v\in W_x$. This proves \eqref{eq:outside-g-neighbours}.

Next,
\begin{equation}
 N_{\Gamma}(x,z)\cap V_{\mathrm{slim}}(\mathfrak g)
 =N_D(x)\cap F_{f_y}.
 \label{eq:inside-g-neighbours}
\end{equation}
Let $u$ belong to the left-hand side. If $f_u=f_x$, then the adjacency
of $u$ and $z$ across two factors implies $z\in C_x$, a contradiction.
Thus $f_u\neq f_x$. Since $u\sim_{\Gamma}x$ and $u,x$ have distinct
fat neighbors, $u\in N_D(x)$. If $f_u\neq f_y$, then
\eqref{eq:disjoint-quasi-cliques} gives $C(f_u)\cap C_y=\varnothing$,
whereas the adjacency of $u$ and $z$ across two factors implies
$z\in C(f_u)$ and then $z\in C(f_u)\cap C_y$. Thus $f_u=f_y$ which proves the inclusion from left to right.

Conversely, let $u\in N_D(x)\cap F_{f_y}$. Then $u\sim_{\Gamma}x$.
Moreover, $u,z\in C_y$, so $u\sim_{\Gamma}z$. This proves
\eqref{eq:inside-g-neighbours}. Combining
\eqref{eq:outside-g-neighbours} and \eqref{eq:inside-g-neighbours}
gives \eqref{eq:common-neighbour-decomposition};
\eqref{eq:internal-common-count} follows from sesqui-regularity.
\end{proof}

\begin{lemma}\label{lem:retention}
Under Assumptions~\ref{ass:Gamma-h}
and~\ref{ass:exceptional-factor}, for any two vertices $x$ and $y$ adjacent in $D$, 
\begin{equation}
 |N_D(x)\cap F_{f_y}|\geq
 \begin{cases}
  2,&c\in\{3,4\},\\
  3,&c\in\{5,6,7,8\}.
 \end{cases}
 \label{eq:neighbor-expansion}
\end{equation}
\end{lemma}

\begin{proof}
Fix a pair of vertices $x$ and $y$ which are adjacent in $D$, and put
\[
 m:=|N_D(x)\cap F_{f_y}|,
 \quad
 m':=|N_D(y)\cap F_{f_x}|.
\]
Both $m$ and $m'$ are positive. In $\Gamma$, every vertex of $Z_{x,y}$ has exactly
$c-m$ neighbors in $C_x$ by
\eqref{eq:internal-common-count}. These neighbors lie  outside $\mathfrak{g}$ by \eqref{eq:outside-g-neighbours}, and hence they belong to $Z_{y,x}=C_x\setminus F_{f_x}$. Interchanging $x$ and $y$, shows that every vertex of $Z_{y,x}$ has exactly
$c-m'$ neighbors in  $Z_{x,y}=C_y\setminus F_{f_y}$. Consequently,
\begin{equation}
 (c-m)|Z_{x,y}|=(c-m')|Z_{y,x}|.
 \label{eq:two-way-edge-count}
\end{equation}

Suppose first that $c\geq5$ and, contrary to the desired conclusion,
$m\leq2$. Choose $u\in Z_{y,x}$. Every vertex of $C_x$ has at most
$35$ neighbors outside $C_x$ by \eqref{eq:clique order}, whereas
$|Z_{x,y}|\geq k-35\geq63$. Hence some $z\in Z_{x,y}$ is nonadjacent
to $u$. Since $|N_{\Gamma}(z)\cap C_x|=c-m\geq1$, the vertices $u$ and $z$ have a common neighbor; hence  $\dist_{\G}(u,z)=2$. The two disjoint sets
\[
 N_{\Gamma}(z)\cap C_x
 \quad\text{and}\quad
 N_{\Gamma}(u)\cap Z_{x,y}
\]
consist of common neighbors of $u$ and $z$. Their sizes are
$c-m$ and $c-m'$, respectively. Sesqui-regularity therefore gives
\[
 (c-m)+(c-m')\leq c,
\]
and therefore
\begin{equation}
 m+m'\geq c.
 \label{eq:two-way-degree-sum}
\end{equation}
The left-hand side of \eqref{eq:two-way-edge-count} is positive, so
$m'\leq c-1$. Moreover, by \eqref{eq:ordered-edge-set-bounds3}, 
\[
|Z_{y,x}|=k+1-d_D(x)-|F_{f_x}|
 \leq k+1-m-m',
\]
because $d_D(x)\geq m$ and $|F_{f_x}|\geq m'$. Thus
\begin{equation}
 (c-m)(k-35)
 \leq(c-m')(k+1-m-m').
 \label{eq:retention-numerical}
\end{equation}
If $m=1$, then \eqref{eq:two-way-degree-sum} and $m'\leq c-1$ give
$m'=c-1$; \eqref{eq:retention-numerical} then gives
$k\leq45$ when $c\in\{5,6,7,8\}$. If $m=2$,
then $m'\in\{c-2,c-1\}$, and \eqref{eq:retention-numerical}  gives $k\leq97$. Both alternatives contradict $k\geq98$.
Therefore $m\geq3$ whenever $c\geq5$.

Finally, suppose that $c\in\{3,4\}$ and $m=1$. The preceding argument
again gives $m+m'\geq c$ and $m'\leq c-1$, so $m'=c-1$. Here
\eqref{eq:retention-numerical} gives $k\leq68$ for $c=3$ and
$k\leq51$ for $c=4$, again a contradiction. Hence $m\geq2$ in this
case. This proves \eqref{eq:neighbor-expansion}.
\end{proof}

The next two lemmas hold throughout the range $3\leq c\leq 8$. They describe the interaction between the independent sets $F_f$ and show that the special graph $D$ is regular of even valency. These structural facts will be used later only in the case $c\in\{3,4\}$.

\begin{lemma}
\label{lem:two-set-cycles} 
Under Assumptions~\ref{ass:Gamma-h}
and~\ref{ass:exceptional-factor}, let
\(f_i\neq f_j\) be fat vertices of
\(\mathfrak g\). Every nontrivial
connected component of \(D[F_{f_i}\cup F_{f_j}]\) is a cycle of even length.
Consequently, every vertex of \(D\) has either zero or two neighbors in
each set \(F_f\) not containing it.
\end{lemma}

\begin{proof}
Let \(H\) be a nontrivial connected component of
\(D[F_{f_i}\cup F_{f_j}]\) with minimum valency $\delta(H)$ and largest
eigenvalue $\lambda_{\max}(H)$. Every vertex of \(H\) is incident with an
edge of \(H\). Lemma~\ref{lem:retention} shows that each vertex of \(H\)
has at least two neighbors. Thus \(\delta(H)\geq2\). Since \(H\) is
an induced subgraph of \(D\), Lemma~\ref{lem:interlacing} and \eqref{eq:Sp-D} give
\[
  \lambda_{\min}(H)\geq\lambda_{\min}(D)\geq-2.
\]
The graph \(H\) is bipartite by Lemma~\ref{lem:special-graph-of-g}
{\rm (ii)} $(e)$, and hence
\(\lambda_{\max}(H)=-\lambda_{\min}(H)\leq2\). On the other hand, the Rayleigh
quotient of the all-ones vector yields
\[
 \lambda_{\max}(H)\geq\frac{2|E(H)|}{|V(H)|}\geq\delta(H)\geq2.
\]
Equality holds throughout. Thus the average valency and
the minimum valency of \(H\) are both equal to \(2\), so \(H\) is a
cycle; since it is bipartite, the cycle has even length. The final assertion
follows immediately.
\end{proof}

\begin{lemma}
\label{lem:D-regular}
Under Assumptions~\ref{ass:Gamma-h}
and~\ref{ass:exceptional-factor}, the graph
\(D\) is regular of even valency.
\end{lemma}

\begin{proof}
Fix a fat vertex $f_i$ of $\mathfrak g$ such that
$F_{f_i}\neq\varnothing$. \eqref{eq:ambient-valency} shows that
$d_D(x)=k+1-|C(f_i)|$ for every $x\in F_{f_i}$. Thus $d_D(x)$ is
constant on $F_{f_i}$; denote its value by $d_{f_i}$.

Suppose that \(D[F_{f_i}\cup F_{f_j}]\) has an edge joining $F_{f_i}$ and $F_{f_j}$. Let $H$ be the component of \(D[F_{f_i}\cup F_{f_j}]\) containing that edge. By Lemma~\ref{lem:two-set-cycles}, \(H\) is a cycle of even length. Let
\[
  X:=V(H)\cap F_{f_i},\qquad Y:=V(H)\cap F_{f_j}.
\]
Define \(\mathbf{q}\in\mathbb R^{V(D)}\) coordinatewise by
\[
 \mathbf{q}_v=
 \begin{cases}
  1,&v\in X,\\
  -1,&v\in Y,\\
  0,&v\notin V(H).
 \end{cases}
\]
Since \(H\) is a cycle of even length with bipartition \((X,Y)\),
\(|X|=|Y|\) and \(|E(H)|=|X|+|Y|\). Therefore
\[
  \mathbf{q}^{\mathsf T}(A(D)+2I)\mathbf{q}
  =-2|E(H)|+2(|X|+|Y|)=0.
\]
By \eqref{eq:Sp-D}, the matrix \(A(D)+2I\) is positive semidefinite;
hence \((A(D)+2I)\mathbf{q}=0\). Multiplying by the all-ones row vector gives
\[
  0=\mathbf{1}^{\mathsf T}(A(D)+2I)\mathbf{q}
  =|X|(d_{f_i}+2)-|Y|(d_{f_j}+2)
  =|X|(d_{f_i}-d_{f_j}).
\]
Thus \(d_{f_i}=d_{f_j}\) whenever an edge of \(D\) joins \(F_{f_i}\) and
\(F_{f_j}\). The auxiliary graph whose vertices are the nonempty sets
\(F_f\), with two such sets adjacent whenever  there is an edge in $D$ joining them, is connected because \(D\) is connected. Hence all values \(d_f\)
are equal, and \(D\) is regular.
Finally, Lemma~\ref{lem:two-set-cycles} shows that every set \(F_f\)
not containing a given vertex contributes either zero or two of its
neighbors, so the valency of $D$ is even.
\end{proof}

\begin{lemma}
\label{lem:local-CP} Under Assumptions~\ref{ass:Gamma-h}
and~\ref{ass:exceptional-factor}, suppose that \(c\in\{3,4\}\) and the
valency of \(D\) is \(2r\). Then \(r\geq1\), and for
every \(u\in V(D)\),
\[
  D[N_D(u)]\cong K_{2,\ldots,2}=\operatorname{CP}(r).
\]
\end{lemma}

\begin{proof}
If \(r=0\), then the connected graph \(D\) is \(K_1\), which is
$1$-integrable, contrary to Lemma~\ref{lem:special-graph-of-g}
{\rm (ii)} $(d)$. Thus \(r\geq1\). For each
\(u\in V(D)\), Lemma~\ref{lem:two-set-cycles} partitions its neighborhood
into independent pairs contained in distinct sets \(F_f\):
\[
  N_D(u)=P_1\,\dot\cup\cdots\dot\cup P_r,
  \qquad |P_i|=2.
\]
It remains only to show that vertices belonging to different pairs are
adjacent. Choose \(y\in P_j\) and \(z\in P_\ell\) with \(j\neq\ell\).
Then \(f_u,f_y,f_z\) are pairwise distinct.

Let \(H\) be the component of
\(D[F_{f_u}\cup F_{f_z}]\) containing the edge \(\{u,z\}\). Put
\[
  U:=V(H)\cap F_{f_u},\qquad Z:=V(H)\cap F_{f_z},
\]
with \(u\in U\). Define the vector \(\mathbf{p}\in\mathbb R^{V(D)}\) 
coordinatewise by
\[
 \mathbf{p}_v=
 \begin{cases}
  1,&v\in U,\\
  -1,&v\in Z,\\
  0,&v\notin V(H).
 \end{cases}
\]
Since \(H\) is a cycle of even length with bipartition \((U,Z)\),
\(|E(H)|=|U|+|Z|\). Hence
\[
 \mathbf{p}^{\mathsf T}(A(D)+2I)\mathbf{p}
 =-2|E(H)|+2(|U|+|Z|)=0.
\]
By \eqref{eq:Sp-D}, the matrix \(A(D)+2I\) is positive semidefinite.
Therefore the preceding equality implies
\[
 (A(D)+2I)\mathbf{p}=0.
\]
The vertex \(y\) lies outside \(U\cup Z\). Taking the \(y\)-coordinate
gives
\[
  |N_D(y)\cap U|=|N_D(y)\cap Z|.
\]
The left-hand side is positive because \(y\sim_D u\) and \(u\in U\).
Hence \(y\) has a neighbor in \(F_{f_z}\), and
Lemma~\ref{lem:two-set-cycles} shows that it has exactly two neighbors
there. By symmetry, \(z\) has exactly two neighbors in \(F_{f_y}\).

Assume that \(y\not\sim_D z\). Since \(y,z\) have distinct unique fat
neighbors, any slim adjacency between them would be a positive edge of
$D$. Thus \(y\) and \(z\) are nonadjacent in \(\G\), while
$\dist_\G(y,z)=2$ through \(u\). Each of the two vertices in
\(N_D(y)\cap F_{f_z}\) is adjacent in \(\Gamma\) to both \(y,z\);
similarly, the two vertices in \(N_D(z)\cap F_{f_y}\) are common
neighbors of \(y\) and \(z\). These four vertices are distinct and are also
distinct from \(u\), because they belong to the three pairwise disjoint
sets \(F_{f_y},F_{f_z},F_{f_u}\). Hence \(y,z\) have at least five
common neighbors in \(\Gamma\), contrary to \(c\leq4\). Therefore
\(y\sim_D z\), proving the claim.
\end{proof}

\subsection{The case \texorpdfstring{\(c\in\{5,6,7,8\}\)}{c at least 5}}

This subsection treats the range $c\in\{5,6,7,8\}$. 

\begin{lemma}\label{lem:two-set-obstruction}
Let \(X\) be a graph whose vertex set is partitioned into independent
sets. Suppose that, for any two adjacent vertices \(u\) and \(v\) in $X$, the vertex
\(u\) has at least three neighbors in the member of the partition
containing \(v\). If \(X\) has an edge, then
\[
  \lambda_{\min}(X)\leq -3.
\]
\end{lemma}

\begin{proof}
Choose two members \(F_i,F_j\) of the partition such that
\(X[F_i\cup F_j]\) contains an
edge, and let \(H\) be a nontrivial connected component of this induced
bipartite graph with minimum valency $\delta(H)$ and largest eigenvalue $\lambda_{\max}(H)$. Every vertex \(z\in V(H)\) is incident with an edge
\(\{z,w\}\) of \(H\). By hypothesis, \(z\) has at least three neighbors in the
set containing \(w\); all of those neighbors belong to \(H\). Hence
\[
  \delta(H)\geq 3.
\]
Because \(H\) is bipartite, its spectrum is symmetric about zero, and thus
\(\lambda_{\min}(H)=-\lambda_{\max}(H)\). The Rayleigh quotient of the all-ones
vector gives
\[
  \lambda_{\max}(H)\geq \frac{2|E(H)|}{|V(H)|}
  \geq \delta(H)\geq 3.
\]
Therefore \(\lambda_{\min}(H)\leq -3\). Since \(H\) is an induced
subgraph of \(X\), eigenvalue interlacing yields
\[
  \lambda_{\min}(X)\leq \lambda_{\min}(H)\leq -3.
\]
\end{proof}

\begin{proposition}
\label{prop:c-ge-5}
Under Assumption~\ref{ass:Gamma-h}, if \(c\geq 5\),
then every indecomposable factor of \(\mathfrak h\) admits an integral
reduced representation of norm \(3\).
\end{proposition}

\begin{proof}
Suppose that an indecomposable factor \(\mathfrak g\) admits no such
representation. Then \(\mathfrak g\), together with its special graph
\(D\), satisfies Assumption~\ref{ass:exceptional-factor}, and
Lemma~\ref{lem:special-graph-of-g} applies. For any two vertices $x$ and $y$ adjacent in $D$, Lemma~\ref{lem:retention}  gives
\[
 |N_D(x)\cap F_{f_y}|\geq3,
 \qquad
 |N_D(y)\cap F_{f_x}|\geq3.
\]
If $D$ had an edge, Lemma~\ref{lem:two-set-obstruction} would imply
\(\lambda_{\min}(D)\leq -3\), contradicting
\eqref{eq:Sp-D}. Hence \(D\) is edgeless. Since \(D\) is connected,
\(D\cong K_1\). This graph is $1$-integrable, contradicting
Lemma~\ref{lem:special-graph-of-g} {\rm (ii)} $(d)$. Therefore no
exceptional factor exists.
\end{proof}

\subsection{The case \texorpdfstring{\(c\in\{3,4\}\)}{c equal to 3 or 4}}
This subsection treats the remaining case $c\in\{3,4\}$. The regularity and local structure established in Lemma \ref{lem:D-regular} and Lemma \ref{lem:local-CP} reduce the problem to the classification of connected locally cocktail-party graphs.

\begin{lemma}
\label{lem:locally-CP-classification}
Let \(X\) be a connected \(2r\)-regular graph such that
\[
  X[N_X(u)]\cong \operatorname{CP}(r)
  \qquad\text{for every }u\in V(X).
\]
Then the following hold.
\begin{enumerate}
\item If \(r=1\), then \(X\) is a cycle of length at least \(4\).
\item If \(r\geq 2\), then \(X\cong\operatorname{CP}(r+1)\).
\end{enumerate}
In either case, $X$ is $1$-integrable.
\end{lemma}

\begin{proof}
Suppose first that \(r=1\). Then \(X\) is connected and \(2\)-regular,
so it is a cycle. The two neighbors of each vertex are nonadjacent in the
local graph, and hence the cycle has length at least \(4\). If its vertices
are indexed cyclically as \(x_1,\ldots,x_m\), define
\[
 \varphi(x_i):=\mathbf{e}_i+\mathbf{e}_{i+1}
  \qquad (1\leq i\leq m),
\]
where indices are taken modulo \(m\). These integral vectors have Gram
matrix \(A(X)+2I\) and thus $X$ is $1$-integrable.

Now assume \(r\geq 2\), and fix \(u\in V(X)\). Let
\(P_1,\ldots,P_r\) be the nonadjacent pairs in
\(X[N_X(u)]\cong\operatorname{CP}(r)\). If \(x\in P_i\), then within
\(\{u\}\cup N_X(u)\) the vertex \(x\) is adjacent to \(u\) and to the
\(2r-2\) vertices of \(N_X(u)\setminus P_i\). Since \(d_X(x)=2r\), the
vertex \(x\) has a unique neighbor outside \(\{u\}\cup N_X(u)\); denote it by
\(x'\). Necessarily \(x'\in X_2(u)\). Conversely, every vertex of \(X_2(u)\) is the unique neighbor outside \(\{u\}\cup N_X(u)\) of some vertex in $N_X(u)$. 

For an arbitrary vertex $x\in N_X(u)$, the vertices \(u\) and \(x'\) are
nonadjacent in the local graph at \(x\). Every vertex of \(\operatorname{CP}(r)\) has a unique
nonneighbor, so \(x'\) is adjacent to every vertex of
\(N_X(x)\setminus\{u,x'\}\). In particular,
\(x'\) is adjacent to \(x\) and to all \(2r-2\) vertices of
\(N_X(u)\setminus P_i\). Thus every vertex of \(X_2(u)\) has at least
\(2r-1\) neighbors in \(N_X(u)\).

Every vertex of \(N_X(u)\) has exactly one neighbor in \(X_2(u)\), so
\[
  e_X\bigl(N_X(u),X_2(u)\bigr)=2r.
\]
Conversely, since \(X_2(u)\neq\varnothing\), while
\[
  2(2r-1)>2r
\]
for \(r\geq 2\), the preceding edge count forces
\(X_2(u)=\{\bar u\}\). It follows that \(\bar u\) is adjacent to all
\(2r\) vertices of \(N_X(u)\), and hence
\(N_X(\bar u)=N_X(u)\).

The set \(\{u,\bar u\}\cup N_X(u)\) is closed under adjacency. Indeed,
if \(x\in P_i\), then all neighbors of \(x\) are
\[
  u,\quad \bar u,\quad\text{and the vertices of }N_X(u)\setminus P_i.
\]
Connectedness therefore implies
\(V(X)=\{u,\bar u\}\cup N_X(u)\). The only nonedges are
\(\{u,\bar u\}\) and the \(r\) pairs \(P_1,\ldots,P_r\), so
\(X\cong\operatorname{CP}(r+1)\).

Finally, index the nonadjacent pairs of \(\operatorname{CP}(r+1)\) by
\(\{x_i^+,x_i^-\}\), \(1\leq i\leq r+1\), and define
\[
  \varphi(x_i^+):=\mathbf{e}_0+\mathbf{e}_i,
  \qquad
   \varphi(x_i^-):=\mathbf{e}_0-\mathbf{e}_i.
\]
These integral vectors have Gram matrix \(A(X)+2I\) and $X$ is $1$-integrable.
\end{proof}

\begin{proposition}
\label{prop:c-3-4}
Under Assumption~\ref{ass:Gamma-h}, if
\(c\in\{3,4\}\), then every indecomposable factor of \(\mathfrak h\)
admits an integral reduced representation of norm \(3\).
\end{proposition}

\begin{proof}
Suppose that an indecomposable factor \(\mathfrak g\) admits no integral
reduced representation of norm \(3\), and let \(D\) be its special graph.
Then \(\mathfrak g\) and \(D\) satisfy
Assumption~\ref{ass:exceptional-factor}. Lemma~\ref{lem:D-regular}
and Lemma~\ref{lem:local-CP} show that \(D\) satisfies the hypotheses of
Lemma~\ref{lem:locally-CP-classification}. Hence $D$ is $1$-integrable,
contradicting Lemma~\ref{lem:special-graph-of-g} {\rm (ii)} $(d)$.
Therefore no exceptional factor exists.
\end{proof}

\begin{proof}[Proof of Theorem~\ref{thm:main}]
Set
\[
  \kappa_{4}
  :=\max\{\kappa_3,\kappa_5,98\},
\]
where \(\kappa_{3}\) and \(\kappa_5\) are the constants from
Theorem~\ref{thm:cgeq9} and Theorem~\ref{thm:hoffman-cover}, respectively. Let
\(\Gamma\) be a connected
sesqui-regular graph with parameters \((n,k,c)\) satisfying
\[
  -3\leq\lambda_{\min}(\Gamma)<-2,
  \qquad c\geq 3,
  \qquad k\geq\kappa_{3}.
\]

If \(c\geq9\), Theorem \ref{thm:extreme} and Theorem~\ref{thm:cgeq9} show that \(\Gamma\) is \(1\)-integrable. It remains to consider
\(3\leq c\leq8\).

Because \(\lambda_{\min}(\Gamma)<-2\), the graph \(\Gamma\) is not
complete. Since it is connected, there exist vertices \(x,y\) at
distance \(2\). Since $|N_{\Gamma}(x)\cap N_{\Gamma}(y)|=c$, the set
$N_{\Gamma}(y)\setminus(\{x\}\cup N_{\Gamma}(x))$ has
$k-c$ vertices. Together with $y$, all these vertices lie outside
$\{x\}\cup N_{\Gamma}(x)$. Hence
\[
 n-k-1
 =|V(\Gamma)\setminus(\{x\}\cup N_{\Gamma}(x))|
 \geq k-c+1
 \geq91>2.
\]
Theorem~\ref{thm:hoffman-cover} therefore supplies a fat Hoffman graph
\(\mathfrak h\) whose slim graph is \(\Gamma\), whose smallest eigenvalue
is at least \(-3\), and whose quasi-cliques are cliques. Hence
Assumption~\ref{ass:Gamma-h} is satisfied.

Propositions~\ref{prop:c-ge-5} and~\ref{prop:c-3-4} together show that
every indecomposable factor of \(\mathfrak h\) has an integral reduced
representation of norm \(3\). Taking the orthogonal direct sum of these
representations gives an integral reduced representation of norm \(3\)
for \(\mathfrak h\). By Theorem~\ref{thm:representation-equivalence},
\(\mathfrak h\) has an integral representation of norm \(3\).
Restricting that representation to the slim vertices yields an integral
matrix \(N\) such that
\[
  A(\Gamma)+3I=N^{\mathsf T}N.
\]
Finally, \(-3\leq\lambda_{\min}(\Gamma)<-2\) implies
\(\lceil-\lambda_{\min}(\Gamma)\rceil=3\). Hence \(\Gamma\) is
\(1\)-integrable.
\end{proof}

\section{Proof of Theorem~\ref{thm:sharpness}}\label{sec:sharpness}
Recall the Cayley graph realization of the Shrikhande graph 
\[
\Sh=\operatorname{Cay}\left(\mathbb{Z}_4^2,~\{\pm(1,0),~\pm(0,1),~\pm(1,1)\}\right).
\]
This graph is strongly regular with parameters $(16,6,2,2)$ and has spectrum 
\begin{equation}\label{spectrum}
\left\{
6^{(1)},2^{(6)},(-2)^{(9)}\right\}.
\end{equation}
See \cite{Shrikhande1959} and \cite[Chapter 10]{BrouwerCohenNeumaier1989}. The proof of Theorem \ref{thm:sharpness} requires lattice results.

\subsection{Lattice preliminaries}

Only the lattice facts used in the proof are recalled here. The
conventions follow \cite{ConwaySloane,Bourbaki}.

Let $\mathbb{R}^n$ be an $n$-dimensional Euclidean space equipped with the standard inner product. A \emph{lattice} $L$ in $\mathbb{R}^n$ is a discrete additive subgroup of  $\mathbb{R}^n$. Given vectors $\mathbf{x}_1,\ldots,\mathbf{x}_r\in L$, if every vector $\mathbf{x}$ in $L$ can be expressed as $\mathbf{x}=\sum_{i=1}^ra_i \mathbf{x}_i$ for some integers $a_1,\ldots,a_r$, then we say $L$ is generated by $\mathbf{x}_1,\ldots,\mathbf{x}_r\in L$. In particular, if $\mathbf{x}_1,\ldots,\mathbf{x}_r\in L$ are linearly independent, then $\mathbf{x}_1,\ldots,\mathbf{x}_r$ are called a \emph{basis} of $L$ and 

\[
 \operatorname{rank}(L):=r
 \quad\text{and}\quad
 \operatorname{Gram}(\mathbf{b}_1,\ldots,\mathbf{b}_r)
 :=\bigl(\ip{\mathbf{b}_i}{\mathbf{b}_j}\bigr)_{i,j=1}^r
\]
are called the \emph{rank} and the \emph{Gram matrix} of $L$, respectively.
The determinant of $L$ is
\[
 \det(L):=\det\operatorname{Gram}(\mathbf{b}_1,\ldots,\mathbf{b}_r).
\]
It is independent of the chosen basis.

A lattice $L$ is \emph{reducible}, if there exist non-zero lattice $L_1$ and lattice $L_2$ such that $(\mathbf x,\mathbf y)=0$ for any $\mathbf x\in L_1$ and  $\mathbf y\in L_2$ and $L=\{\mathbf x+\mathbf y\mid \mathbf x\in L_1,\mathbf y\in  L_2\}$. Otherwise $L$ is \emph{irreducible}.

A lattice is \emph{integral} if $(\mathbf{x},\mathbf{y})\in\Z$ for all $\mathbf x,\mathbf y\in L$, and it is \emph{even} if $\ip{x}{x}\in2\Z$ for all $x\in L$.  

Let $L\subseteq \mathbb{R}^n$ be an integral lattice of rank $r$. Let $L\otimes\mathbb{R}$ be the $r$-dimensional subspace of $\mathbb{R}^n$ spanned by the vectors of $L$. The \emph{dual lattice} of $L$ is
\[
 L^*:=\{\mathbf z\in L\otimes\mathbb{R}:\ip{\mathbf z}{\mathbf x}\in\Z\text{ for every }\mathbf x\in L\}.
\]
It follows $L\subseteq L^*$. For a positive real number $c$, let
\[
cL:=\{c\mathbf x:\mathbf x\in L\}.
\]
Then
\begin{equation}\label{scale lattice}
(cL)^*=c^{-1}L^*.
\end{equation}

A vector \(\mathbf x\) in an integral lattice is a \emph{root} if
\(\ip{\mathbf x}{\mathbf x}=2\); a lattice
generated by its roots is a \emph{root lattice}. Every root lattice is even.

\begin{theorem}[{\cite{Witt1941Spiegelungsgruppen}},{\cite[p.~27]{ConwaySloane1988LowDim}}]
\label{thm:irreducible-root-lattices}
Each irreducible root lattice is isometric to one of the following:
	\begin{enumerate}[label=\rm{(\roman*)}]
		\item $A_n=\{\mathbf{x}\in\mathbb{Z}^{n+1}\mid \sum\limits_{i=1}^{n+1} x_i=0\}$ for $n\geq1$,
		\item $D_n=\{\mathbf{x}\in\mathbb{Z}^n\mid \sum\limits_{i=1}^n x_i$ is even$\}$ for $n\geq4$,
		\item $E_8=D_8\cup (\mathbf{c}+D_8)$, where $\mathbf{c}=\frac{1}{2}(1,1,1,1,1,1,1,1)^{\mathsf T}$,
		\item $E_7=\{\mathbf{x}\in E_8\mid \sum\limits_{i=1}^8 x_i=0\}$,
		\item $E_6=\{\mathbf{x}\in E_7\mid x_7+x_8=0\}$.
	\end{enumerate}
Their ranks and determinants are
\[
\begin{array}{c|ccccc}
 L&A_n&D_n&E_6&E_7&E_8\\ \hline
 \operatorname{rank}(L)&n&n&6&7&8\\
 \det(L)&n+1&4&3&2&1.
\end{array}
\]
\end{theorem}

The following fact about the dual lattice of $E_7$ is also needed.
\begin{lemma}[{\cite[p.~40]{ConwaySloane1988LowDim}}]\label{dual of E7}
The dual lattice of $E_7$ can be expressed as following:
\[
E_7^*=E_7\cup\left\{ \frac{1}{4}(1,1,1,1,1,1,-3,-3)^{\mathsf T}+\mathbf x\mid \mathbf x\in E_7 \right\}.\]

Consequently, for every $\mathbf w\in E_7^*$,
\[
\ip{\mathbf w}{\mathbf w}\in
\begin{cases}
	2\Z,&\mathbf w\in E_7,\\[2mm]
	\dfrac32+2\Z,&\mathbf w\notin E_7.
\end{cases}
 \]
\end{lemma}

\begin{lemma}\label{lem:E7-dual-short}
Let $L\cong \sqrt2 E_7$. If $0\ne \mathbf z\in L^*$ and
$\ip{\mathbf z}{\mathbf z}\le1$, then
\[
 \ip{\mathbf z}{\mathbf z}\in\left\{\frac34,1\right\}.
\]
\end{lemma}

\begin{proof}By \eqref{scale lattice},
\[
 L^*\cong\frac1{\sqrt2}E_7^*.
\]
Write \(\mathbf{z}=\mathbf{w}/\sqrt2\), where
\(\mathbf{w}\in E_7^*\).  If \(\mathbf{w}\in E_7\), then
Lemma~\ref{dual of E7} gives
\[
 \ip{\mathbf{z}}{\mathbf{z}}
 =\frac12\ip{\mathbf{w}}{\mathbf{w}}\in\mathbb Z.
\]
If \(\mathbf{w}\notin E_7\), then the same lemma gives
\[
 \ip{\mathbf{z}}{\mathbf{z}}
 =\frac12\ip{\mathbf{w}}{\mathbf{w}}
 \in\frac34+\mathbb Z.
\]
The assertion follows from
\(0<\ip{\mathbf{z}}{\mathbf{z}}\leq1\).
\end{proof}

 \subsection{The integrability of the graph $\Sh\square K_t$}

To prove Theorem~\ref{thm:sharpness}, first we show that the graph $\Sh\square K_t$ is sesqui-regular and has smallest eigenvalue $-3$.

Given two matrix $M$ and $R$, define the \emph{Kronecker product}  $M\otimes R$ of $M$ and $R$ to be the matrix  by replacing the $ij$-entry of $M$ by $M_{ij}R$ for all $i$ and $j$. 

\begin{lemma}[{\cite[p.~206]{GodsilRoyle2001}}]\label{lem:eigenvalue cartesian product}
Let $G$ and $H$ be two graphs with adjacency matrix $A(G)$ and $A(H)$ respectively and let $G\square H$ be the Cartesian product of $G$ and $H$. If $A(G\square H)$ is the adjacency matrix of $G\square H$, then
\[
A(G\square H)=A(G)\otimes I+I\times A(H).
\]
In particular, if $\theta$ is an eigenvalue of $G$ with multiplicity $m_1$ and $\tau$ is an eigenvalue of $H$ with multiplicity $m_2$, then $\theta+\tau$ is an eigenvalue of the graph $G\square H$ with multiplicity $m_1m_2$.
\end{lemma}

\begin{lemma}\label{lem:Gt is sesqui-regular}
For every $t\geq2$, the graph $\G_t=\Sh \square K_t$is connected and sesqui-regular with parameters $(n,k,c)=(16t,t+5,2)$ and $\lmin(\G_t)=-3$.
\end{lemma}
\begin{proof}
Both $\Sh$ and $K_t$ are connected, so is \(\Gamma_t\). The vertex set of \(\Gamma_t\)  is $V(\Gamma_t)=V(\Gamma_{\rm Sh})\times\{1,2,\ldots,t\}$,  and therefore $|V(\Gamma_t)|=16t$. Every vertex $(x,i)\in V(\G_t)$ has six neighbors in the set $\{(y,i)\mid y\in V(\Sh)\}$ and \(t-1\)
neighbors in the set $\{(x,j)\mid j=1,2,\ldots,t\}$. Hence the valency of $\Gamma_t$ is $t+5$. 

Let \((x,i)\) and \((y,j)\) be vertices at distance two.  If $i=j$,
then \(\operatorname{dist}_{\Sh}(x,y)=2\), and their common neighbours
are the two vertices \((z,i)\), where $z$ is a common neighbour of $x$
and $y$ in $\Sh$.  If $i\ne j$, then $x\sim_{\Sh}y$, and the common
neighbours are precisely \((x,j)\) and \((y,i)\).  Hence every pair of
vertices at distance two has exactly two common neighbours.

The spectra of $\Sh$ and $K_t$ are
\[
 \left\{6^{(1)},2^{(6)},(-2)^{(9)}\right\}
 \quad\text{and}\quad
 \left\{(t-1)^{(1)},(-1)^{(t-1)}\right\},
\]
respectively.  Lemma~\ref{lem:eigenvalue cartesian product} therefore
gives the spectrum
\[
 \left\{(t+5)^{(1)},(t+1)^{(6)},(t-3)^{(9)},
 5^{(t-1)},1^{(6t-6)},(-3)^{(9t-9)}\right\}
\]
of $\G_t$, and so \(\lmin(\G_t)=-3\).

This completes the proof.
\end{proof}

\begin{theorem}\label{thm:ShK2-not-integrable-short}
The graph $\Sh\square K_2$ is not $1$-integrable.
\end{theorem}

\begin{proof}
Put $A=A(\Sh)$ and suppose, to the contrary, that
\(\Sh\square K_2\) is $1$-integrable.  By
Lemma~\ref{lem:Gt is sesqui-regular}, its smallest eigenvalue is $-3$.
Thus there are integral vectors
\[
 \{\mathbf{u}_x,\mathbf{v}_x:x\in V(\Sh)\}\subset\mathbf{R}^m
\]
such that
\begin{equation}\label{eq:ShK2-gram}
 \ip{\mathbf{u}_x}{\mathbf{u}_y}
 =\ip{\mathbf{v}_x}{\mathbf{v}_y}=(A+3I)_{xy},
 \qquad
 \ip{\mathbf{u}_x}{\mathbf{v}_y}=\delta_{xy}
 \quad(x,y\in V(\Sh)).
\end{equation}
Set
\[
 \mathbf{a}_x=\mathbf{u}_x-\mathbf{v}_x,
 \qquad
 \mathbf{b}_x=\mathbf{u}_x+\mathbf{v}_x.
\]
It follows from \eqref{eq:ShK2-gram} that
\begin{equation}\label{eq:ShK2-sum-diff}
 \ip{\mathbf{a}_x}{\mathbf{a}_y}=2(A+2I)_{xy},
 \qquad
 \ip{\mathbf{a}_x}{\mathbf{b}_y}=0
 \quad(x,y\in V(\Sh)).
\end{equation}
Let $W$ be the vector space spanned by the vector set $
\{\mathbf{a}_x:x\in V(\Sh)\}$. The spectrum in \eqref{spectrum} and \eqref{eq:ShK2-sum-diff} give
\begin{equation}\label{eq:ShK2-W-dimension}
 \dim W=\operatorname{rank}(A+2I)=7.
\end{equation}
Moreover, every \(\mathbf{b}_x\) is orthogonal to $W$.

Choose the seven vertices
\[
 (0,0),(0,1),(0,2),(1,0),(1,2),(2,0),(3,3)
\]
of the Cayley model of $\Sh$, in this order denoted by
\(x_1,\ldots,x_7\).  Let $\Lambda$ be the lattice generated by
\(\mathbf{a}_{x_1},\ldots,\mathbf{a}_{x_7}\), and put
\(\boldsymbol{\alpha}_i=\frac{1}{\sqrt2}\mathbf{a}_{x_i}\).  The Gram matrix of
the vectors \(\boldsymbol{\alpha}_1,\ldots,
\boldsymbol{\alpha}_7\) is
\begin{equation}\label{eq:ShK2-E7-gram}
M=
 \begin{pmatrix}
 2&1&0&1&0&0&1\\
 1&2&1&0&1&0&0\\
 0&1&2&0&1&0&0\\
 1&0&0&2&0&1&0\\
 0&1&1&0&2&0&0\\
 0&0&0&1&0&2&0\\
 1&0&0&0&0&0&2
 \end{pmatrix},
 \qquad \det M=2.
\end{equation}
Thus \(\frac{1}{\sqrt2}\Lambda\) is an even root lattice of rank $7$ and
determinant $2$. The nonorthogonality graph of the displayed roots is
connected, so this root lattice is irreducible.
Theorem~\ref{thm:irreducible-root-lattices}
therefore yields $\frac{1}{\sqrt2}\Lambda \cong E_7$, equivalently, 
\begin{equation}\label{eq:ShK2-Lambda-E7}
 \Lambda\cong\sqrt2E_7.
\end{equation}
Since \(\det M\ne0\), the seven selected vectors are independent; by
\eqref{eq:ShK2-W-dimension}, they span $W$.

Let \(\mathbf{e}_1,\ldots,\mathbf{e}_m\) be the standard orthonormal
basis of $\mathbb{R}^m$, and put
\[
 \mathbf{p}_i=\pi_W(\mathbf{e}_i),
\]
where $\pi_W$ is orthogonal projection onto $W$.  For every
\(\mathbf{y}\in\Lambda\),
\[
 \ip{\mathbf{p}_i}{\mathbf{y}}
 =\ip{\mathbf{e}_i}{\mathbf{y}}\in\mathbb Z,
\]
because $\Lambda$ is generated by integral vectors.  Hence
\(\mathbf{p}_i\in\Lambda^*\), and
\(\ip{\mathbf{p}_i}{\mathbf{p}_i}\leq1\).  By
Lemma~\ref{lem:E7-dual-short},
\begin{equation}\label{eq:ShK2-projection-values}
 \mathbf{p}_i=\mathbf{0}
 \quad\text{or}\quad
 \ip{\mathbf{p}_i}{\mathbf{p}_i}\in\left\{\frac34,1\right\}.
\end{equation}

We next exclude the value $1$.  Suppose that
\(\ip{\mathbf{p}_i}{\mathbf{p}_i}=1\).  Since
\(\mathbf{e}_i-\mathbf{p}_i\) is orthogonal to $W$,
\[
 1=\ip{\mathbf{e}_i}{\mathbf{e}_i}
 =\ip{\mathbf{p}_i}{\mathbf{p}_i}
 +\ip{\mathbf{e}_i-\mathbf{p}_i}{\mathbf{e}_i-\mathbf{p}_i},
\]
and therefore \(\mathbf{e}_i=\mathbf{p}_i\in W\).  Thus
\(\ip{\mathbf{e}_i}{\mathbf{b}_x}=0\) for every $x\in V(\Sh)$.  Using
the definitions of \(\mathbf{a}_x\) and \(\mathbf{b}_x\), we obtain
\[
 \ip{\mathbf{e}_i}{\mathbf{a}_x}
 =2\ip{\mathbf{e}_i}{\mathbf{u}_x}\in2\mathbb Z
 \qquad(x\in V(\Sh)).
\]
Consequently,
\(\ip{\frac12\mathbf{e}_i}{\mathbf{y}}\in\mathbb Z\) for every
\(\mathbf{y}\in\Lambda\), and hence
\[
 \frac12\mathbf{p}_i=\frac12\mathbf{e}_i\in\Lambda^*.
\]
However, \(\ip{\frac12\mathbf{p}_i}{\frac12\mathbf{p}_i}=\ip{\frac12\mathbf{e}_i}{\frac12\mathbf{e}_i}=\frac{1}{4}\), contradicting
Lemma~\ref{lem:E7-dual-short}.  Therefore
\begin{equation}\label{eq:ShK2-projection-norms}
 \ip{\mathbf{p}_i}{\mathbf{p}_i}\in\left\{0,\frac34\right\}
 \qquad(1\leq i\leq m).
\end{equation}

Finally, the trace of the orthogonal projection onto $W$ is $\dim W=7$.
Its diagonal entries in the basis
\(\mathbf{e}_1,\ldots,\mathbf{e}_m\) are
\(\ip{\mathbf{p}_i}{\mathbf{p}_i}\), and hence
\[
 7=\sum_{i=1}^m\ip{\mathbf{p}_i}{\mathbf{p}_i}.
\]
This contradicts \eqref{eq:ShK2-projection-norms}, since the right-hand
side belongs to \(\frac{3}{4}\mathbb Z\), whereas \(7\notin\frac{3}{4}\mathbb Z\).
This contradiction proves the theorem.
\end{proof}

\begin{proof}[Proof of Theorem~\ref{thm:sharpness}]
Lemma~\ref{lem:Gt is sesqui-regular} proves the asserted connectedness,
parameters, and smallest eigenvalue.  It remains to prove that $\G_t$ is
not $1$-integrable.  Note that $\G_t$ contains $\G_2=\Sh\square K_2$ as an induced subgraph, which is not $1$-integrable. It follows immediately that 
$\G_t$ is also not $1$-integrable.
\end{proof}

\section*{Acknowledgements}
Q. Yang is supported by  the National Natural Science Foundation of China (No. 12401460).

\end{document}